\documentclass[sn-mathphys,smallextended]{sn-jnl}

\RequirePackage{fix-cm}

\usepackage{graphicx}
\usepackage{multirow}
\usepackage{subcaption}
\usepackage{anyfontsize}
\newcommand{\snorm}[1]{\|#1\|^2}

\jyear{2022}

\theoremstyle{thmstyleone}

\theoremstyle{thmstyletwo}

\newtheorem{remark}{Remark}

\theoremstyle{thmstylethree}

\newtheorem{lemma}{Lemma}[section]

\begin{document}

\title{Non-conforming interface conditions for the second-order wave equation}

\author{\fnm{Gustav} \sur{Eriksson}}\email{gustav.eriksson@it.uu.se}

\affil{\orgdiv{Department of Information Technology}, \orgname{Uppsala University}, \orgaddress{\street{PO Box 337}, \city{Uppsala}, \postcode{S-751 05}, \country{Sweden}}}

\abstract{Imposition methods of interface conditions for the second-order wave equation with non-conforming grids is considered. The spatial discretization is based on high order finite differences with summation-by-parts properties. Previously presented solution methods for this problem, based on the simultaneous approximation term (SAT) method, have shown to introduce significant stiffness. This can lead to highly inefficient schemes. Here, two new methods of imposing the interface conditions to avoid the stiffness problems are presented: 1) a projection method and 2) a hybrid between the projection method and the SAT method. Numerical experiments are performed using traditional and order-preserving interpolation operators. Both of the novel methods retain the accuracy and convergence behavior of the previously developed SAT method but are significantly less stiff.}


\maketitle

\section{Introduction}
It is well known that high order finite differences are highly efficient for large-scale wave propagation problems \cite{Kreiss1972}. However, the design of such schemes requires particular care at the boundaries to obtain stability. One way to obtain stable high order finite difference schemes is to use finite difference operators with a \emph{summation-by-parts} (SBP) property together with simultaneous-approximation-terms (SBP-SAT) \cite{Carpenter1994}, the projection method (SBP-P) \cite{Olsson1995a,Olsson1995} or ghost points (SBP-GP) \cite{Sjogreen2012}. SBP finite difference operators are essentially standard finite difference stencils in the interior with boundary closures carefully designed to mimic integration by parts in the discrete setting. The SBP difference operators have an associated discrete inner product such that a discrete energy equation that is analogous to the continuous equation can be derived. The boundary conditions should be imposed such that the scheme exhibits no non-physical energy growth, sometimes referred to as \emph{strict stability} \cite{Gustafsson642577}. The SAT method achieves this by adding penalty terms that weakly impose the boundary conditions such that the resulting scheme is stable, see for example \cite{DelReyFernandez2014}. The SBP-GP method adds ghost points at the boundaries and computes their values such that the boundary conditions are imposed and the scheme is stable \cite{Petersson2015,Wang2019}. The projection method derives an orthogonal projection and rewrites the problem such that it is solved in the subspace of solutions where the boundary conditions are exactly fulfilled, see \cite{Mattsson2018}. 

An important aspect of finite difference methods is the ability to split the computational domain into blocks and couple them across the interfaces. This is necessary to handle complex geometries, but also to increase the efficiency of the schemes. For example, in the case of the wave equation, a finer grid spacing is only needed in regions of the domain where the wave speed is high. In other regions, a coarser grid may be used. In general, the grid points at each side of an interface are non-conforming, in which case interpolations are used to couple the solutions. In the framework of SBP finite differences, it is crucial that the method of imposing the interpolated interface conditions preserves the SBP properties of the difference operators.

The construction of interpolation operators along with SATs to obtain stable schemes with non-conforming interfaces has received significant attention in the past \cite{Mattsson2010,Kozdon2016,Wang2016}. In \cite{Mattsson2010} so-called SBP-preserving interpolation operators (here referred to as norm-compatible) were first constructed and used to derive stable schemes for general hyperbolic and parabolic problems. However, it was noted in \cite{Wang2016,Wang2018} that the global convergence rate was decreased by one (compared to the convergence rate with conforming grids) for problems involving second derivatives in space. In \cite{Almquist2019} this is solved by constructing order-preserving (OP) interpolation operators along with SATs such that the global convergence rate is preserved. The new operators come in two norm-compatible pairs (a pair consists of one restriction operator and one prolongation operator), where one of the operators in each pair is of one order higher accuracy. Using both pairs, an SAT is presented in \cite{Almquist2019} where the first interface condition (continuity of the solution) is imposed using the accurate interpolation and the second interface condition (continuity of the first derivative) using the less accurate interpolation.

A major downside of the SBP-SAT discretizations is the necessary decomposition of the second derivative SBP operator to obtain an energy estimate. Often referred to as the "borrowing trick" \cite{Mattsson2008}. This procedure is known to introduce additional stiffness to the problem, especially for large wave speed discontinuities. The main contribution of the current work is two new methods avoiding this problem, one using SBP-P and the other a hybrid SBP-P-SAT. The analysis and numerical experiments are done on the second-order wave equation. However, the discrete Laplace operators presented are equally applicable to the heat equation and the Schrödinger equation. There are indications that the new methods can be applied to other problems, such as first-order hyperbolic systems, but this is out of the scope of the current work.

The paper is structured as follows: In Section \ref{sec: defs} some necessary definitions and the discrete operators are introduced. In Section \ref{sec: cont_analysis} the continuous problem is presented. The new semi-discrete schemes are presented in Section \ref{sec: disc_analysis}. The time discretization is presented in Section \ref{sec: time}. In Section \ref{sec: num_exp} numerical experiments validating the new methods and comparing them to the SBP-SAT schemes are presented. Conclusions are drawn in Section \ref{sec: conclusions}.
\section{Definitions}
\label{sec: defs}
Let
\begin{equation}
	\label{eq: cont_innerprod}
	(u,v) = \int _\Omega u v \: dx \quad \text{and} \quad \snorm{u} = (u,u),
\end{equation}
define an inner product and the corresponding norm for functions $u,v$ on a rectangular domain $\Omega$. The domain is split across the $x$-axis into a left and a right block, denoted $\Omega_L$ and $\Omega_R$. The two blocks are discretized using $m_x^{(u,v)}$ and $m_y^{(u,v)}$ equidistant grid points in the $x$- and $y$-directions respectively.

The second-derivatives in each block and direction are approximated using one-dimensional SBP finite difference operators \cite{Mattsson2004} satisfying
\begin{equation}
	\label{eq: 1dsbp}
	D_2 = H^{-1}(-M + e_r d_r^\top - e_l d_l^\top),
\end{equation}
where $H$ is diagonal and positive definite, $M$ is symmetric and positive semi-definite, $e_{l,r}^\top$ are row-vectors extracting the solution at the first and last grid points and $d_{l,r}^\top$ are row-vectors approximating the first derivative of the solution at the first and last grid points. The matrix $D_2$ is referred to as a $2pth$-order accurate second derivative SBP operator. In the interior $D_2$ consists of a $2p$ order accurate central finite difference stencil. On the boundaries, for the SBP properties to hold with a diagonal $H$, the order of accuracy is limited to $p$. Thus, the theoretical global order of accuracy with these operators is $\min(2p,p+2)$ \cite{Svard2019}. In this paper, numerical results are presented for 4th and 6th order SBP operators. Hence, the expected convergence rates are $4$ and $5$.

The matrix $H$ defines a one-dimensional discrete inner product and norm as
\begin{equation}
	(u,v)_H = u^\top H v \quad \text{and} \quad \snorm{u}_H = (u,u)_H.
\end{equation}
The one-dimensional operators are extended to two dimensions using Kronecker products as follows:
\begin{equation}
\begin{alignedat}{5}
	D_{2x} &= (D_2 \otimes I_{m_y}), \quad &&D_{2y} = (I_{m_x} \otimes D_2), \\
	H_x &= (H \otimes I_{m_y}), &&H_y = (I_{m_x} \otimes H), \\
	M_x &= (M \otimes I_{m_y}), &&M_y = (I_{m_x} \otimes M), \\
	e_W &= (e_l^\top \otimes I_{m_y}), &&e_E = (e_r^\top \otimes I_{m_y}), \\
	e_S &= (I_{m_x} \otimes e_l^\top), &&e_N = (I_{m_x} \otimes e_r^\top), \\
	d_W &= (d_l^\top \otimes I_{m_y}), &&d_E = (d_r^\top \otimes I_{m_y}), \\
	d_S &= (I_{m_x} \otimes d_l^\top), &&d_N = (I_{m_x} \otimes d_r^\top), \\
\end{alignedat}	
\end{equation}
where $I_m$ denotes the $m \times m$ identity matrix. The discrete inner product and norm over the 2D domain is given by
\begin{equation}
	\label{eq: disc_innerprod}
	(u,v)_{\bar H} = u^\top {\bar H} v \quad \text{and} \quad \snorm{u}_{\bar H} = (u,u)_{\bar H},
\end{equation}
where $\bar H = H_x H_y$. The discrete Laplace operator is given by
\begin{equation}
	D_L = D_{2x} + D_{2y}.
\end{equation}
Using the SBP properties \eqref{eq: 1dsbp}, the discrete two-dimensional Laplace operator can be written as
\begin{equation}
	D_L = H_x^{-1} (-M_x + e_E^\top d_E - e_W^\top d_W) + H_y^{-1} (-M_y + e_N^\top d_N - e_S^\top d_S),
\end{equation}
or for two vectors $u_{1,2} \in \mathbb{R}^{m_x m_y}$ we have
\begin{equation}
	\label{eq: disc_lapl_sbp}
	\begin{alignedat}{2}
		(u_1,D_L u_2)_{\bar H} &= -u_1^\top (H_y M_x + H_x M_y) u_2 + (e_E u_1, d_E u_2)_H - (e_W u_1, d_W u_2)_H \\
		& + (e_N u_1, d_N u_2)_H - (e_S u_1, d_S u_2)_H.
	\end{alignedat}
\end{equation}
In the upcoming analysis, let the solutions in the left block be denoted by $u$, and in the right block by $v$. Superscripts $(u)$ and $(v)$ will be used to denote which block an operator belongs to. For example, the inner-product matrix $\bar H^{(u)}$ acts on solution vectors in the left block, with $m^{(u)}_x m^{(u)}_y$ unknowns.
\subsection{Interpolation operators} 
Interpolation operators are used at the interface to couple two blocks with non-conforming grid points. Let $I_{u2v}$ denote the operator interpolating from left to right, and $I_{v2u}$ the operator interpolating from right to left. See Figure \ref{fig: domain}. For stability, we require that the pair of operators are \emph{norm-compatible}, i.e. they must satisfy
\begin{equation}
	\label{eq: disc_inter_sbppres}
	(I_{v2u} v, u)_{H^{(u)}} = (v,I_{u2v} u)_{H^{(v)}}, \quad \forall u \in \mathbb{R}^{m_y^{(u)}}, v \in \mathbb{R}^{m_y^{(v)}}.
\end{equation}	
Note that the additional constraint known as \emph{norm-contracting} \cite{Wang2016} is not needed here.

In this paper interpolation operators for a 1:2 grid ratio corresponding to 4th and 6th SBP operators \cite{Mattsson2004} are used. The traditional interpolation operators derived in \cite{Mattsson2010} are compared to the OP operators \cite{Almquist2019}. The OP operators come in two pairs: $I_{u2v}^{b}$ and $I_{v2u}^{g}$, and $I_{u2v}^{g}$ and $I_{v2u}^{b}$, where the "good" operators (superscript $g$) are one order more accurate than the "bad" operators (superscript $b$). Each pair of the OP operators satisfies \eqref{eq: disc_inter_sbppres}.
\begin{figure}[!htbp]
	\centering
	\includegraphics[width=0.6\textwidth]{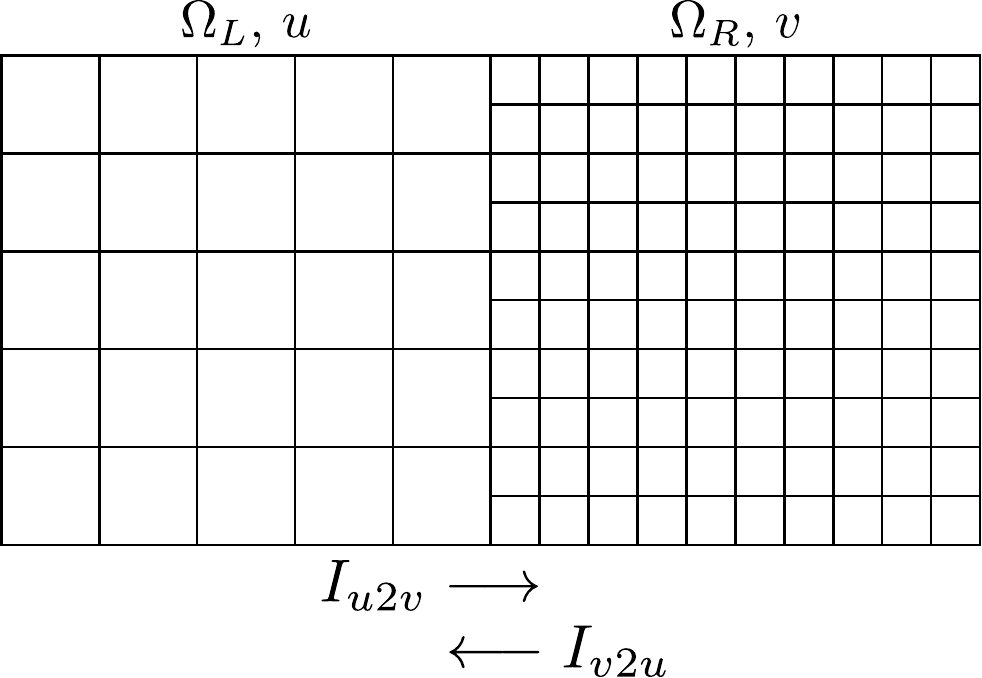}
	\caption{A two block domain with a 1:2 grid ratio non-conforming interface.}
	\label{fig: domain}
\end{figure}
\section{Continuous analysis}
\label{sec: cont_analysis}
We consider the initial-value boundary problem
\begin{equation}
		\label{eq: cont_waveeq}
	\begin{alignedat}{4}
		u_{tt} &= c_1^2 \Delta u, \quad &&(x,y) \in \Omega_L, &&&t \geq 0, \\
		v_{tt} &= c_2^2 \Delta v, &&(x,y) \in \Omega_R, &&&t \geq 0, \\
		n \cdot \nabla u &= g_u, && (x,y) \in \partial \Omega_L \setminus \partial \Omega_I, \quad &&& t \geq 0, \\
		n \cdot \nabla v &= g_v, && (x,y) \in \partial \Omega_R \setminus \partial \Omega_I, \quad &&& t \geq 0, \\
		u &= v, &&(x,y) \in \partial \Omega_I, &&&t \geq 0, \\
		c_1^2 u_x &= c_2^2 v_x, &&(x,y) \in \partial \Omega_I, &&&t \geq 0, \\
	\end{alignedat}
\end{equation}
with initial data for $u$, $u_t$, $v$, and $v_t$ at $t = 0$. Here $\partial \Omega_{L,R}$ denote the boundaries of the blocks, $\partial \Omega_I$ denotes the interface, $n$ is the outward pointing normal, $g_{u,v}$ are boundary data, and $c_1$ and $c_2$ are real, positive constants.

Multiplying the first equation in \eqref{eq: cont_waveeq} by $u_t$ and integrating over $\Omega_L$, the second equation by $v_t$ and integrating over $\Omega_R$, adding the results and using integration by parts leads to the energy equation
\begin{equation}
	\label{eq: cont_energy}
	\frac{d}{dt} E = 2c_1^2 \int _{\partial \Omega_L} n \cdot \nabla u u_t \: dS + 2c_2^2 \int _{\partial \Omega_R} n \cdot \nabla v v_t \: dS.
\end{equation}
The energy is given by
\begin{equation}
	E = \snorm{u_t} + \snorm{v_t} + c_1^2 \snorm{\nabla u} + c_2^2 \snorm{\nabla v},
\end{equation}
Inserting the interface and boundary conditions (the last four equations in \eqref{eq: cont_waveeq}) and assuming $g_{u,v} = 0$ leads to energy conservation,
\begin{equation}
	\frac{d}{dt} E = 0.
\end{equation}
This energy estimate is sufficient to show that \eqref{eq: cont_waveeq} is stable and has a unique solution.
\section{Spatial discretization}
We now turn to the spatial discretization, time is left continuous. For completeness, the boundary treatment of a single block using SBP-SAT is first presented in Section \ref{sec: singleblock}. Then, in Section \ref{sec: multiblock}, the novel discretizations of the multi-block problem \eqref{eq: cont_waveeq} are presented.
\label{sec: disc_analysis}
\subsection{Single-block analysis}
\label{sec: singleblock}
Consider the initial-value boundary problem on the rectangular two-dimensional domain $\Omega$ given by
\begin{equation}
	\label{eq: cont_waveeq_1block}
	\begin{alignedat}{4}
		u_{tt} &= c^2 \Delta u, \quad &&(x,y) \in \Omega, &&&t \geq 0, \\
		n \cdot \nabla u &= g, && (x,y) \in \partial \Omega, \quad &&& t \geq 0, \\
	\end{alignedat}
\end{equation}
with initial data for $u$ and $u_t$. Discretize $\Omega$ into a Cartesian grid and let $v$ denote a column-major ordered semi-discrete solution vector. A consistent semi-discrete approximation of \eqref{eq: cont_waveeq_1block} with boundary condition imposed using the SAT method \cite{Mattsson2009} is given by
\begin{equation}
	\label{eq: disc_single_block_ode}
	\begin{alignedat}{2}
		v_{tt} = c^2 D_L u &+ c^2 \bar H_x^{-1} e_W^\top (d_W u - g_W) - c^2 \bar H_x^{-1} e_E^\top (d_E u - g_E) \\
		&+ c^2 \bar H_y^{-1} e_S^\top (d_S u - g_S) - c^2 \bar H_y^{-1} e_N^\top (d_N u - g_N),
	\end{alignedat}
\end{equation}
where $g_{W,E,S,N}$ are vectors of $g$ evaluated on the boundary grid points. Taking the inner product \eqref{eq: disc_innerprod} between $v_t$ and \eqref{eq: disc_single_block_ode}, and using \eqref{eq: disc_lapl_sbp} results in
\begin{equation}
\begin{alignedat}{2}
		(v_t,v_{tt})_{\bar H} = -c^2 v_t^\top (H_y M_x + H_x M_y) v &- c^2 (e_W v_t,g_W)_H + c^2 (e_E v_t,g_E) \\
		&- c^2 (e_S v_t, g_S)_H + c^2 (e_N v_t, g_N)_H.
\end{alignedat}
\end{equation}
Setting $g_{W,E,S,N} = 0$ and adding the transpose leads to the energy equation
\begin{equation}
	\label{eq: 1block_enest}
	\frac{d}{dt} E = 0,
\end{equation}
where
\begin{equation}
	\label{eq: 1block_E}
	E = ||v_t||_{\bar H} + c^2 v^\top (H_y M_x + H_x M_y) v.
\end{equation}
Since $E \geq 0$ it defines an energy, and the energy equation \eqref{eq: 1block_enest} shows that it is conserved over time.
\subsection{Multi-block analysis}
\label{sec: multiblock}
We now consider the multi-block problem \eqref{eq: cont_waveeq}. To make the analysis more readable, it is assumed that the boundary conditions in both blocks are treated as described in Section \ref{sec: singleblock}. Thus, the terms corresponding to outer boundaries are left out. 

Denote by $u$ and $v$ column-major ordered semi-discrete solution vectors in the left and right blocks respectively. Let $w = \begin{bmatrix}
	u \\ v
\end{bmatrix}$ be the global semi-discrete solution vector. Discretizing \eqref{eq: cont_waveeq} in space without imposing the interface conditions yields
\begin{equation}
	\label{eq: disc_waveeq}
		\begin{alignedat}{2}
			w_{tt} &= D w, \\
			L w &= 0,
	\end{alignedat}
\end{equation}
where
\begin{equation}
	\label{eq: disc_DL}
	D = \begin{bmatrix} c_1^2 D_L^{(u)} & 0 \\ 0 & c_2^2 D_L^{(v)}\end{bmatrix},
\end{equation}
and $L$ is a linear operator approximating the interface conditions (for now $L$ is left unspecified).

The interface conditions $Lw = 0$ are imposed using SBP-P-SAT or SBP-P. The resulting problem with both methods can be written as
\begin{equation}
	\label{eq: disc_waveeq_ode}
	w_{tt} = P \tilde D P w,
\end{equation}
where 
\begin{equation}
	\label{eq: mod_SBP}
	\tilde D = D + SAT,
\end{equation}
 is a modified spatial operator and $P$ is a projection operator given by
\begin{equation}
	\label{eq: disc_P_op}
	P = I - \hat H^{-1} L^\top (L \hat H^{-1} L^\top)^{-1} L.
\end{equation}
By construction, $P$ is the orthogonal projection operator with respect to the global inner product $(\cdot,\cdot)_{\hat H}$, where
\begin{equation}
	\label{eq: disc_full_H}
	\hat H = \begin{bmatrix} \bar{H}^{(u)} & 0 \\ 0 & \bar{H}^{(v)} \end{bmatrix},
\end{equation}
i.e., it satisfies the self-adjoint property
\begin{equation}
	\label{eq: disc_p_selfadjoint}
	(w_1,Pw_2)_{\hat{H}} = (Pw_1,w_2)_{\hat{H}}, \quad \forall w_{1,2} \in \mathbb{R}^{m_x^{(u)} m_y^{(u)} + m_x^{(v)} m_y^{(v)}}.
\end{equation}
For more details on the projection method see \cite{Olsson1995a,Olsson1995,Mattsson2018}. See also \cite{Mattsson2006} for examples of the projection method used for interface conditions.

\subsection{Stability with SBP-P}
\label{subsec: sbp-p stab}
We begin by considering only the projection method to impose the interface conditions, where $SAT = 0$. Let $\hat w = \begin{bmatrix} \hat u \\ \hat v \end{bmatrix} = P w$ denote the projected solution vector. Taking the inner product between $w_t$ and \eqref{eq: disc_waveeq_ode}, and using \eqref{eq: disc_lapl_sbp} and \eqref{eq: disc_p_selfadjoint} leads to
\begin{equation}
	\label{eq: disc_enest}
	\frac{d}{dt} E = 2 c_1^2 (e_E \hat u_t, d_E \hat u)_{H^{(u)}} - 2 c_2^2 (e_W \hat v_t, d_W \hat v)_{H^{(v)}},
\end{equation}
where $E$ is an energy given by
\begin{equation}
	\label{eq: disc_energy}
	\begin{alignedat}{2}
		E &= \snorm{u_t}_{\bar H^{(u)}} + c_1^2 \hat u^\top (H_y^{(u)} M_x^{(u)} + H_x^{(u)} M_y^{(u)}) \hat u  \\
		 &+ \snorm{v_t}_{\bar H^{(v)}} + c_2^2 \hat v^\top (H_y^{(v)} M_x^{(v)} + H_x^{(v)} M_y^{(v)}) \hat v \geq 0.
	\end{alignedat}
\end{equation}	
The semi-discrete energy equation \eqref{eq: disc_enest} is the discrete analog to \eqref{eq: cont_energy}. The following lemma is the first main result of this paper:
\begin{lemma}
	The ODE \eqref{eq: disc_waveeq_ode} with $L$ given by
	\begin{equation}
	\label{eq: disc_nonconf_L_proj}
		L = \begin{bmatrix}
		e_E & -I_{v2u} e_W \\ c_1^2 I_{u2v} d_E^{(u)} & -c_2^2 d_W^{(v)}
	\end{bmatrix},
	\end{equation}
	is a stable approximation of \eqref{eq: cont_waveeq} with a non-conforming interface.
\end{lemma}
\begin{proof}
	Since $L \hat w = L P w = 0$, we have 
	\begin{equation}
		e_E \hat u = I_{v2u} e_W \hat v \quad \text{and} \quad c_2^2 d_W^{(v)} \hat v = c_1^2 I_{u2v} d_E^{(u)} \hat u.
	\end{equation} 
	Substituted into \eqref{eq: disc_enest} results in
	\begin{equation}
		\label{eq: disc_nonconfen}
		\frac{d}{dt} E = 2 c_1^2 ((I_{v2u} e_W \hat v_t, d_E^{(u)} \hat u)_{H^{(u)}} - (e_W \hat v_t, I_{u2v} d_E^{(u)} \hat u)_{H^{(v)}}).
	\end{equation}
	Using that the interpolation operators are norm-compatible, i.e., that they satisfy \eqref{eq: disc_inter_sbppres}, we get
	\begin{equation}
		\frac{d}{dt} E = 0,
	\end{equation}
	which proves stability.
\end{proof}
\subsection{Stability with SBP-P-SAT}
\label{subsec: sbp-p-sat stab}
With the hybrid method, the continuity of the first derivative is imposed using a SAT given by
\begin{equation}
	\label{eq: disc_SAT}
	SAT = \begin{bmatrix}
		0 & 0 \\ -(H_x^{(v)})^{-1} e_W^\top c_1^2 I_{u2v} d_E^{(u)} & (H_x^{(v)})^{-1} e_W^\top c_2^2 d_W^{(v)}
	\end{bmatrix}.
\end{equation}
This corresponds to imposing the interface condition
\begin{equation}
	c_2^2 d_W^{(v)} \hat v = c_1^2 I_{u2v} d_E^{(u)} \hat u,
\end{equation}
weakly on the equation for $v$.

Using the modified spatial operator \eqref{eq: mod_SBP}, the energy equation becomes
\begin{equation}
	\label{eq: disc_enest2}
	\frac{d}{dt} E = 2 c_1^2 ((e_E \hat u_t, d_E \hat u)_{H^{(u)}} - (e_W \hat v_t,I_{u2v} d_E \hat u)_{H^{(v)}}),
\end{equation}
where $\hat w = \begin{bmatrix} \hat u \\ \hat v \end{bmatrix} = P w$ and $E$ is given by \eqref{eq: disc_energy}. The following lemma is the second main result of this paper:
\begin{lemma}
	The ODE \eqref{eq: disc_waveeq_ode} with $L$ given by
	\begin{equation}
	\label{eq: disc_nonconf_L_hyb}
		L = \begin{bmatrix}
		e_E & -I_{v2u} e_W
	\end{bmatrix},
	\end{equation}
	and $SAT$ given by \eqref{eq: disc_SAT} is a stable approximation of \eqref{eq: cont_waveeq} with a non-conforming interface.
\end{lemma}
\begin{proof}
	Since $L \hat w = L P w = 0$ we have 
	\begin{equation}
		e_E \hat u = I_{v2u} e_W \hat v.
	\end{equation} 
	Substituted into \eqref{eq: disc_enest2} results in
	\begin{equation}
		\label{eq: disc_nonconfen}
		\frac{d}{dt} E = 2 c_1^2 ((I_{v2u} e_W \hat v_t, d_E \hat u)_{H^{(u)}} - (e_W \hat v_t,I_{u2v} d_E \hat u)_{H^{(v)}}).
	\end{equation}
	Using that the interpolation operators are norm-compatible, i.e., that they satisfy \eqref{eq: disc_inter_sbppres}, we get
	\begin{equation}
		\frac{d}{dt} E = 0,
	\end{equation}
	which proves stability.
\end{proof}
\begin{remark}
	With both SBP-P and SBP-P-SAT the key to obtaining energy stability is the mirrored interpolations of the interface conditions. In Sections \ref{subsec: sbp-p stab} and \ref{subsec: sbp-p-sat stab}, continuity of the solution is imposed by interpolating right to left and the continuity of the first derivative by interpolating left to right. Conservative energy estimates can also be obtained by swapping the interpolations and using the transpose of \eqref{eq: disc_inter_sbppres}. With only projection we have
	\begin{equation}
	L = \begin{bmatrix}
		I_{u2v} e_E & -e_W \\ c_1^2 d_E^{(u)} & -c_2^2 I_{v2u} d_W^{(v)}
	\end{bmatrix},
	\end{equation}
	and with the hybrid method
	\begin{equation}
	L = \begin{bmatrix}
		I_{u2v} e_E & -e_W
	\end{bmatrix},
	\end{equation}
	and 
	\begin{equation}
	\label{eq: disc_SAT2}
	SAT = \begin{bmatrix}
		 -(H_x^{(u)})^{-1} e_E^\top c_1^2 d_E & (H_x^{(u)})^{-1} e_E^\top c_2^2 I_{v2u} d_W\\ 0 & 0
	\end{bmatrix}.
\end{equation}
	Numerical experiments have shown that the differences between the choices in terms of accuracy and stiffness are minor, and dependent on the SBP and interpolation operators used. The results presented in this paper are obtained using the discretizations presented in Sections \ref{subsec: sbp-p stab} and \ref{subsec: sbp-p-sat stab}.
\end{remark}
\subsection{Order preserving interpolation}
Taking inspiration from \cite{Almquist2019}, we note that continuity of the solution should be imposed using the "good" interpolation operators and that continuity of the first derivative should be imposed using the "bad" operators. For the discretizations in Sections \ref{subsec: sbp-p stab} and \ref{subsec: sbp-p-sat stab}, this amounts to replacing $I_{v2u}$ with $I_{v2u}^g$ and $I_{u2v}$ with $I_{u2v}^b$ in \eqref{eq: disc_nonconf_L_proj}, \eqref{eq: disc_SAT}, and \eqref{eq: disc_nonconf_L_hyb}. Note that with SBP-P and SBP-P-SAT only one pair of the OP interpolation operators is used, whereas the SBP-SAT discretization requires both pairs.
\section{Time discretization}
\label{sec: time}
All methods considered can be written as a system of second-order ODEs, given by
\begin{equation}
	\label{eq: ODE_system}
	\begin{alignedat}{2}
		w_{tt} &= Q w + G(t), \quad &&t > 0,\\
		w(t) &= f_1, && t = 0, \\
		w_t(t) &= f_2, && t = 0, \\
	\end{alignedat}
\end{equation}
where $Q$ is a matrix approximating the spatial derivatives including boundary and interface conditions and $G(t)$ contains the boundary data. In this paper \eqref{eq: ODE_system} is solved using an explicit 4th order time-marching scheme \cite{Mattsson2006}, given by
\begin{equation}
	\begin{alignedat}{1}
		w^{(0)} &= f_1, \\
		w^{(1)} &= (I + \frac{k^2}{2} Q) f_1 + k(I + \frac{k^2}{6} Q) f_2 + \frac{k^2}{2} G(0) + \frac{k^3}{6}G_t(0), \\
		w^{(n+1)} &= (2I + k^2 Q + \frac{k^4}{12} Q^2) w^{(n)} - w^{(n-1)} \\
		&+ k^2 (I + \frac{k^2}{12} Q) G(t_n) + \frac{k^4}{12} G_{tt}(t_n),
	\end{alignedat}
\end{equation}
where $I$ is the identity matrix, $k$ denotes the time step, and $t_n = nk$, $n = 0,1,...$, is the discrete time-level. It can be shown that the scheme is stable if 
\begin{equation}
	k^2 \rho(Q) < 12,
\end{equation}
where $\rho(Q)$ denotes the spectral radius of $Q$. Introducing the undivided matrix $\tilde Q = h^2 Q$, where $h$ is the spatial interval, we get the stability condition
\begin{equation}
	k < \sqrt{\frac{12}{\rho(\tilde Q)}} h.
\end{equation}
The scaled spectral radius $\rho(\tilde Q)$ depends on the discretization method, but not on the spatial interval $h$ (for large enough problems). Therefore, comparing the scaled spectral radius of the methods gives a good indication of the required time steps, and consequently the overall efficiency of the schemes. 
\section{Numerical experiments}
\label{sec: num_exp}
In this section numerical experiments are presented comparing the new discretizations to the SAT discretizations presented in \cite{Wang2018} (traditional interpolation) and \cite{Almquist2019} (OP interpolation). Neumann boundary conditions imposed using the SAT method are used for all results, as described in Section \ref{sec: singleblock}. The domain is given by $[-10,10] \times [0,10]$ with an interface at $x = 0$. The left and right blocks are discretized with $m$ and $2m - 1$ grid points in each dimension.

The methods are compared in terms of efficiency (measured by the spectral radius) in Section \ref{sec: specrad}, and accuracy for a problem with a known analytical solution in Section \ref{sec: accuracy}.
\subsection{Spectral radius}
\label{sec: specrad}
In Table \ref{tabl: specrad} the scaled spectral radius of the SBP-P, SBP-P-SAT, and SBP-SAT schemes are presented for the 4th and 6th order SBP operators with traditional and order-preserving interpolation operators. As a reference, the scaled spectral radius of the single-block discretization \eqref{eq: disc_single_block_ode} is also included. 

With all operators, the spectral radius with SBP-P and SBP-P-SAT are significantly lower than with SBP-SAT. Furthermore, the scaled spectral radii obtained with SBP-P and SBP-P-SAT are the same as for the single-block discretization with Neumann boundary conditions. This shows that the spectral radius with SBP-P and SBP-P-SAT is unaffected by the interface coupling procedure. This is not true for SBP-SAT. As an illustration, for the 6th order OP operators with a given grid resolution, approximately 2.5 times larger time steps can be used with SBP-P or SBP-P-SAT compared to SBP-SAT. For the 4th order OP operators, the ratio is approximately 6.6.
\begin{remark}
	The SBP-SAT schemes involve tuning the value of a parameter. Typically, increasing its value leads to a more accurate scheme (up to a point) at the cost of increasing the spectral radius. How to choose this parameter is not obvious, and one unclear aspect of the SAT method. The results in this paper are obtained using the same values as in \cite{Wang2018} and \cite{Almquist2019}. 
\end{remark}

\begin{table}
\center
\caption{Scaled spectral radius of RHS matrix with order- and non-order-preserving interpolation operators with projection (SBP-P), hybrid projection and SAT (SBP-P-SAT), and SAT (SBP-SAT) for 4th and 6th order SBP operators. The final column shows the scaled spectral radius of the single-block discretization with 4th and 6th order SBP operators.}
\label{tabl: specrad}
\begin{tabular}{|c|c|c|c|c|} 
\hline
Operators & SBP-P & SBP-P-SAT & SBP-SAT & Single-block\\
\hline
Traditional 4th order & 10.66 & 10.66 & 57.21 & \multirow{2}{*}{10.66} \\
Order-preserving 4th order & 10.66 & 10.66 & 467.82 & \\
\hline
Traditional 6th order & 28.36 & 28.36 & 133.41 & \multirow{2}{*}{28.36} \\
Order-preserving 6th order & 28.36 & 28.36 & 180.81 & \\
\hline
\end{tabular}
\end{table}

\subsection{Accuracy}
\label{sec: accuracy}
In this section the accuracy of the methods is compared using an analytical solution given by
\begin{equation}
	\label{eq: cont_analyticsol}
	\begin{alignedat}{2}
		u &= \cos(x + y - \sqrt{2} c_1 t) + k_2 \cos(x - y + \sqrt{2} c_1 t), \\
		v &= (1 + k_2) \cos(k_1 x + y + \sqrt{2} c_1 t),
	\end{alignedat}
\end{equation}
where $k_1 = \sqrt{2 c_1^2/c_2^2-1}$ and $k_2 = (c_1^2 - c_2^2 k_1)/(c_1^2 + c_2^2 k_1)$. The wave speeds are set to $c_1 = 1$ and $c_2 = 0.5$. The boundary and initial data are given by \eqref{eq: cont_analyticsol}. The time step is chosen as one tenth of the largest stable time step (with this choice the temporal errors are insignificant in comparison to the spatial errors). The convergence rate is approximated as
\begin{equation}
	q = \frac{\log(\frac{e_1}{e_2})}{\log(\frac{m_1}{m_2})},
\end{equation}
where $e_1$ and $e_2$ are errors in the $H$-norm \eqref{eq: disc_full_H} at $t = 2$ of two separate simulations with $m = m_1$ and $m = m_2$.

In Table \ref{tabl: errconv} the error and convergence results of the SBP-P, SBP-P-SAT, and SBP-SAT discretizations are presented for the 4th and 6th order traditional and OP interpolation operators. Overall the accuracy of the  SBP-P, SBP-P-SAT, and SBP-SAT schemes are very similar. With the traditional interpolation operators, 3rd and 4th order convergence are obtained with the 4th and 6th order operators respectively. And, with the order-preserving interpolation operators, convergence rates 4 and 5 are obtained. This shows that SBP-P and SBP-P-SAT exhibit the same convergence behaviors as previously observed with SBP-SAT, where the traditional interpolation operators lead to an order reduction whereas the OP interpolation operators retain the full convergence rates. One stand-out result is the accuracy with the 6th order traditional interpolation operators. With SBP-P and SBP-P-SAT, the errors with $m = 801$ are smaller by almost one magnitude compared to the error with SBP-SAT. 

\begin{table}
\caption{Error (in base 10 logarithm) and convergence of 4th and 6th order traditional and order-preserving interpolation and SBP operators with projection (subscript $p$), hybrid (subscript $h$), and SAT (subscript $s$) discretizations.}
\label{tabl: errconv}
\begin{subtable}[h]{1\textwidth}
\centering
\caption{Traditional 4th order}
\begin{tabular}{|c||c|c||c|c||c|c|} 
\hline
$m$ & $e_p$ & $q_p$ & $e_h$ & $q_h$ & $e_s$ & $q_s$   \\ 
\hline
26  & -1.74 & - 	 & -1.75 & - & -1.74 & - 	  \\ 
\hline
51  & -2.97 & -4.18 & -2.98 & -4.18 & -2.93 & -4.05  \\ 
\hline
101 & -4.09 & -3.76 & -4.10 & -3.74 & -3.96 & -3.46  \\ 
\hline
201 & -5.09 & -3.33 & -5.09 & -3.33 & -4.93 & -3.25   \\ 
\hline
401 & -6.02 & -3.10 & -6.02 & -3.10 & -5.89 & -3.20  \\ 
\hline
801 & -6.93 & -3.03 & -6.93 & -3.03 & -6.79 & -2.99  \\ 
\hline
\end{tabular}
\end{subtable}
\begin{subtable}[h]{1\textwidth}
\centering
\vspace{0.3cm}
\caption{Order-preserving 4th order}
\begin{tabular}{|c||c|c||c|c||c|c|} 
\hline
$m$ & $e_p$ & $q_p$ & $e_h$ & $q_h$ & $e_s$ & $q_s$   \\ 
\hline
26  & -1.77 & - 	 & -1.78 & - & -1.78 & -  \\ 
\hline
51  & -3.04 & -4.32 & -3.05 & -4.33 & -3.06 & -4.36  \\ 
\hline
101 & -4.28 & -4.15 & -4.28 & -4.14 & -4.30 & -4.15 \\ 
\hline
201 & -5.51 & -4.12 & -5.52 & -4.12 & -5.54 & -4.14  \\ 
\hline
401 & -6.73 & -4.06 & -6.73 & -4.05 & -6.76 & -4.06  \\ 
\hline
801 & -7.94 & -4.03 & -7.95 & -4.03 & -7.97 & -4.03  \\ 
\hline
\end{tabular}
\end{subtable}
\begin{subtable}[h]{1\textwidth}
\centering
\vspace{0.3cm}
\caption{Traditional 6th order}
\begin{tabular}{|c||c|c||c|c||c|c|} 
\hline
$m$ & $e_p$ & $q_p$ & $e_h$ & $q_h$ & $e_s$ & $q_s$   \\ 
\hline
26 & -1.93 & -  & -1.89 & - & -1.86 & - \\ 
\hline
51 & -3.62 & -5.74 & -3.61 & -5.82 & -3.49 & -5.54  \\ 
\hline
101 & -5.24 & -5.44 & -5.23 & -5.45 & -4.91 & -4.76 \\ 
\hline
201 & -6.79 & -5.19 & -6.79 & -5.23 & -6.17 & -4.20  \\ 
\hline
401  & -8.23 & -4.78 & -8.23 & -4.78 & -7.35 & -3.93 \\ 
\hline
801  & -9.53 & -4.34 & -9.53 & -4.33 & -8.57 & -4.06 \\ 
\hline
\end{tabular}
\end{subtable}
\begin{subtable}[h]{1\textwidth}
\centering
\vspace{0.3cm}
\caption{Order-preserving 6th order}
\begin{tabular}{|c||c|c||c|c||c|c|} 
\hline
$m$ & $e_p$ & $q_p$ & $e_h$ & $q_h$ & $e_s$ & $q_s$   \\ 
\hline
26  & -1.93 & - 	 & -1.89 & - & -1.87 & - \\ 
\hline
51  & -3.63 & -5.78 & -3.62 & -5.86 & -3.59 & -5.83  \\ 
\hline
101 & -5.28 & -5.52 & -5.27 & -5.52 & -5.23 & -5.51 \\ 
\hline
201 & -6.89 & -5.40 & -6.89 & -5.43 & -6.86 & -5.44  \\ 
\hline
401 & -8.48 & -5.28 & -8.48 & -5.29 & -8.47 & -5.36 \\ 
\hline
801 & -9.99 & -5.03 & -10.10 & -5.07 & -10.04 & -5.23 \\ 
\hline
\end{tabular}
\end{subtable}
\end{table}
\section{Conclusions}
\label{sec: conclusions}
Two new SBP finite difference discretizations of the second-order wave equation with non-conforming grid interfaces are presented. The first scheme utilizes the projection method to impose the interface conditions and the second scheme a hybrid projection-SAT method. Energy conservation is shown for both discretizations using the energy method. Numerical experiments with traditional and order-preserving interpolation operators demonstrate similar accuracy and convergence behavior as for the SAT schemes. The most significant advantage of the new methods compared to SAT is the reduced spectral radius of the spatial operators. The new methods are less stiff than the SAT schemes, allowing for several times larger time steps with explicit time integration methods. Furthermore, it is found that the stiffness of the new schemes is the same as without the interface altogether, i.e., it is unaffected by the coupling procedure. Although the analysis and numerical experiments are done for the second-order wave equation, the discrete Laplace operator presented here can be directly applied to the heat equation and the Schrödinger equation. In a future study, the ideas introduced in this paper will be extended to general hyperbolic systems.

\clearpage
\bibliography{references}


\begin{thebibliography}{20}
\ifx \bisbn   \undefined \def \bisbn  #1{ISBN #1}\fi
\ifx \binits  \undefined \def \binits#1{#1}\fi
\ifx \bauthor  \undefined \def \bauthor#1{#1}\fi
\ifx \batitle  \undefined \def \batitle#1{#1}\fi
\ifx \bjtitle  \undefined \def \bjtitle#1{#1}\fi
\ifx \bvolume  \undefined \def \bvolume#1{\textbf{#1}}\fi
\ifx \byear  \undefined \def \byear#1{#1}\fi
\ifx \bissue  \undefined \def \bissue#1{#1}\fi
\ifx \bfpage  \undefined \def \bfpage#1{#1}\fi
\ifx \blpage  \undefined \def \blpage #1{#1}\fi
\ifx \burl  \undefined \def \burl#1{\textsf{#1}}\fi
\ifx \doiurl  \undefined \def \doiurl#1{\url{https://doi.org/#1}}\fi
\ifx \betal  \undefined \def \betal{\textit{et al.}}\fi
\ifx \binstitute  \undefined \def \binstitute#1{#1}\fi
\ifx \binstitutionaled  \undefined \def \binstitutionaled#1{#1}\fi
\ifx \bctitle  \undefined \def \bctitle#1{#1}\fi
\ifx \beditor  \undefined \def \beditor#1{#1}\fi
\ifx \bpublisher  \undefined \def \bpublisher#1{#1}\fi
\ifx \bbtitle  \undefined \def \bbtitle#1{#1}\fi
\ifx \bedition  \undefined \def \bedition#1{#1}\fi
\ifx \bseriesno  \undefined \def \bseriesno#1{#1}\fi
\ifx \blocation  \undefined \def \blocation#1{#1}\fi
\ifx \bsertitle  \undefined \def \bsertitle#1{#1}\fi
\ifx \bsnm \undefined \def \bsnm#1{#1}\fi
\ifx \bsuffix \undefined \def \bsuffix#1{#1}\fi
\ifx \bparticle \undefined \def \bparticle#1{#1}\fi
\ifx \barticle \undefined \def \barticle#1{#1}\fi
\bibcommenthead
\ifx \bconfdate \undefined \def \bconfdate #1{#1}\fi
\ifx \botherref \undefined \def \botherref #1{#1}\fi
\ifx \url \undefined \def \url#1{\textsf{#1}}\fi
\ifx \bchapter \undefined \def \bchapter#1{#1}\fi
\ifx \bbook \undefined \def \bbook#1{#1}\fi
\ifx \bcomment \undefined \def \bcomment#1{#1}\fi
\ifx \oauthor \undefined \def \oauthor#1{#1}\fi
\ifx \citeauthoryear \undefined \def \citeauthoryear#1{#1}\fi
\ifx \endbibitem  \undefined \def \endbibitem {}\fi
\ifx \bconflocation  \undefined \def \bconflocation#1{#1}\fi
\ifx \arxivurl  \undefined \def \arxivurl#1{\textsf{#1}}\fi
\csname PreBibitemsHook\endcsname

\bibitem{Kreiss1972}
\begin{barticle}
\bauthor{\bsnm{Kreiss}, \binits{H.-O.}},
\bauthor{\bsnm{Oliger}, \binits{J.}}:
\batitle{{Comparison of accurate methods for the integration of hyperbolic
  equations}}.
\bjtitle{Tellus}
\bvolume{24}(\bissue{3}),
\bfpage{199}--\blpage{215}
(\byear{1972}).
\doiurl{10.3402/tellusa.v24i3.10634}
\end{barticle}
\endbibitem

\bibitem{Carpenter1994}
\begin{barticle}
\bauthor{\bsnm{Carpenter}, \binits{M.H.}},
\bauthor{\bsnm{Gottlieb}, \binits{D.}},
\bauthor{\bsnm{Abarbanel}, \binits{S.}}:
\batitle{Time-stable boundary conditions for finite-difference schemes solving
  hyperbolic systems: Methodology and application to high-order compact
  schemes}.
\bjtitle{Journal of Computational Physics}
\bvolume{111}(\bissue{2}),
\bfpage{220}--\blpage{236}
(\byear{1994}).
\doiurl{10.1006/jcph.1994.1057}
\end{barticle}
\endbibitem

\bibitem{Olsson1995a}
\begin{barticle}
\bauthor{\bsnm{Olsson}, \binits{P.}}:
\batitle{{Summation by parts, projections, and stability. I}}.
\bjtitle{Mathematics of Computation}
\bvolume{64}(\bissue{211}),
\bfpage{1035}
(\byear{1995}).
\doiurl{10.1090/s0025-5718-1995-1297474-x}
\end{barticle}
\endbibitem

\bibitem{Olsson1995}
\begin{barticle}
\bauthor{\bsnm{Olsson}, \binits{P.}}:
\batitle{{Summation by parts, projections, and stability. II}}.
\bjtitle{Mathematics of Computation}
\bvolume{64}(\bissue{212}),
\bfpage{1473}--\blpage{1473}
(\byear{1995}).
\doiurl{10.1090/s0025-5718-1995-1308459-9}
\end{barticle}
\endbibitem

\bibitem{Sjogreen2012}
\begin{barticle}
\bauthor{\bsnm{Sj{\"{o}}green}, \binits{B.}},
\bauthor{\bsnm{Petersson}, \binits{N.A.}}:
\batitle{{A fourth order accurate finite difference scheme for the elastic wave
  equation in second order formulation}}.
\bjtitle{Journal of Scientific Computing}
\bvolume{52}(\bissue{1}),
\bfpage{17}--\blpage{48}
(\byear{2012}).
\doiurl{10.1007/s10915-011-9531-1}
\end{barticle}
\endbibitem

\bibitem{Gustafsson642577}
\begin{bbook}
\bauthor{\bsnm{Gustafsson}, \binits{B.}},
\bauthor{\bsnm{Kreiss}, \binits{H.-O.}},
\bauthor{\bsnm{Oliger}, \binits{J.}}:
\bbtitle{Time-Dependent Problems and Difference Methods},
\bedition{2}nd edn.
(\byear{2013}).
\doiurl{10.1002/9781118548448}
\end{bbook}
\endbibitem

\bibitem{DelReyFernandez2014}
\begin{botherref}
\oauthor{\bsnm{{Del Rey Fern{\'{a}}ndez}}, \binits{D.C.}},
\oauthor{\bsnm{Hicken}, \binits{J.E.}},
\oauthor{\bsnm{Zingg}, \binits{D.W.}}:
{Review of summation-by-parts operators with simultaneous approximation terms
  for the numerical solution of partial differential equations}
(2014).
\doiurl{10.1016/j.compfluid.2014.02.016}
\end{botherref}
\endbibitem

\bibitem{Petersson2015}
\begin{barticle}
\bauthor{\bsnm{Petersson}, \binits{N.A.}},
\bauthor{\bsnm{Sj{\"{o}}green}, \binits{B.}}:
\batitle{{Wave propagation in anisotropic elastic materials and curvilinear
  coordinates using a summation-by-parts finite difference method}}.
\bjtitle{Journal of Computational Physics}
\bvolume{299},
\bfpage{820}--\blpage{841}
(\byear{2015}).
\doiurl{10.1016/j.jcp.2015.07.023}
\end{barticle}
\endbibitem

\bibitem{Wang2019}
\begin{barticle}
\bauthor{\bsnm{Wang}, \binits{S.}},
\bauthor{\bsnm{Petersson}, \binits{N.A.}}:
\batitle{Fourth order finite difference methods for the wave equation with mesh
  refinement interfaces}.
\bjtitle{SIAM Journal on Scientific Computing}
\bvolume{41}(\bissue{5}),
\bfpage{3246}--\blpage{3275}
(\byear{2019}).
\doiurl{10.1137/18M1211465}
\end{barticle}
\endbibitem

\bibitem{Mattsson2018}
\begin{barticle}
\bauthor{\bsnm{Mattsson}, \binits{K.}},
\bauthor{\bsnm{Almquist}, \binits{M.}},
\bauthor{\bparticle{van~der} \bsnm{Weide}, \binits{E.}}:
\batitle{{Boundary optimized diagonal-norm SBP operators}}.
\bjtitle{Journal of Computational Physics}
\bvolume{374},
\bfpage{1261}--\blpage{1266}
(\byear{2018}).
\doiurl{10.1016/j.jcp.2018.06.010}
\end{barticle}
\endbibitem

\bibitem{Mattsson2010}
\begin{barticle}
\bauthor{\bsnm{Mattsson}, \binits{K.}},
\bauthor{\bsnm{Carpenter}, \binits{M.H.}}:
\batitle{{Stable and accurate interpolation operators for high-order multiblock
  finite difference methods}}.
\bjtitle{SIAM Journal on Scientific Computing}
\bvolume{32}(\bissue{4}),
\bfpage{2298}--\blpage{2320}
(\byear{2010}).
\doiurl{10.1137/090750068}
\end{barticle}
\endbibitem

\bibitem{Kozdon2016}
\begin{barticle}
\bauthor{\bsnm{Kozdon}, \binits{J.E.}},
\bauthor{\bsnm{Wilcox}, \binits{L.C.}}:
\batitle{{Stable coupling of nonconforming, high-order finite difference
  methods}}.
\bjtitle{SIAM Journal on Scientific Computing}
\bvolume{38}(\bissue{2}),
\bfpage{923}--\blpage{952}
(\byear{2016}).
\doiurl{10.1137/15M1022823}
\end{barticle}
\endbibitem

\bibitem{Wang2016}
\begin{barticle}
\bauthor{\bsnm{Wang}, \binits{S.}},
\bauthor{\bsnm{Virta}, \binits{K.}},
\bauthor{\bsnm{Kreiss}, \binits{G.}}:
\batitle{{High order finite difference methods for the wave equation with
  non-conforming grid interfaces}}.
\bjtitle{Journal of Scientific Computing}
\bvolume{68}(\bissue{3}),
\bfpage{1002}--\blpage{1028}
(\byear{2016}).
\doiurl{10.1007/s10915-016-0165-1}
\end{barticle}
\endbibitem

\bibitem{Wang2018}
\begin{barticle}
\bauthor{\bsnm{Wang}, \binits{S.}}:
\batitle{{An improved high order finite difference method for non-conforming
  grid interfaces for the wave equation}}.
\bjtitle{Journal of Scientific Computing}
\bvolume{77}(\bissue{2}),
\bfpage{775}--\blpage{792}
(\byear{2018}).
\doiurl{10.1007/s10915-018-0723-9}
\end{barticle}
\endbibitem

\bibitem{Almquist2019}
\begin{botherref}
\oauthor{\bsnm{Almquist}, \binits{M.}},
\oauthor{\bsnm{Wang}, \binits{S.}},
\oauthor{\bsnm{Werpers}, \binits{J.}}:
{Order-preserving interpolation for summation-by-parts operators at
  nonconforming grid interfaces}.
SIAM Journal on Scientific Computing
\textbf{41}(2)
(2019).
\doiurl{10.1137/18M1191609}
\end{botherref}
\endbibitem

\bibitem{Mattsson2008}
\begin{barticle}
\bauthor{\bsnm{Mattsson}, \binits{K.}},
\bauthor{\bsnm{Ham}, \binits{F.}},
\bauthor{\bsnm{Iaccarino}, \binits{G.}}:
\batitle{{Stable and accurate wave-propagation in discontinuous media}}.
\bjtitle{Journal of Computational Physics}
\bvolume{227}(\bissue{19}),
\bfpage{8753}--\blpage{8767}
(\byear{2008}).
\doiurl{10.1016/j.jcp.2008.06.023}
\end{barticle}
\endbibitem

\bibitem{Mattsson2004}
\begin{barticle}
\bauthor{\bsnm{Mattsson}, \binits{K.}},
\bauthor{\bsnm{Nordstr{\"{o}}m}, \binits{J.}}:
\batitle{{Summation by parts operators for finite difference approximations of
  second derivatives}}.
\bjtitle{Journal of Computational Physics}
\bvolume{199}(\bissue{2}),
\bfpage{503}--\blpage{540}
(\byear{2004}).
\doiurl{10.1016/j.jcp.2004.03.001}
\end{barticle}
\endbibitem

\bibitem{Svard2019}
\begin{barticle}
\bauthor{\bsnm{Sv{\"{a}}rd}, \binits{M.}},
\bauthor{\bsnm{Nordstr{\"{o}}m}, \binits{J.}}:
\batitle{{On the convergence rates of energy-stable finite-difference
  schemes}}.
\bjtitle{Journal of Computational Physics}
\bvolume{397},
\bfpage{108819}
(\byear{2019}).
\doiurl{10.1016/j.jcp.2019.07.018}
\end{barticle}
\endbibitem

\bibitem{Mattsson2009}
\begin{barticle}
\bauthor{\bsnm{Mattsson}, \binits{K.}},
\bauthor{\bsnm{Ham}, \binits{F.}},
\bauthor{\bsnm{Iaccarino}, \binits{G.}}:
\batitle{{Stable boundary treatment for the wave equation on second-order
  form}}.
\bjtitle{Journal of Scientific Computing}
\bvolume{41}(\bissue{3}),
\bfpage{366}--\blpage{383}
(\byear{2009}).
\doiurl{10.1007/s10915-009-9305-1}
\end{barticle}
\endbibitem

\bibitem{Mattsson2006}
\begin{barticle}
\bauthor{\bsnm{Mattsson}, \binits{K.}},
\bauthor{\bsnm{Nordstr{\"{o}}m}, \binits{J.}}:
\batitle{{High order finite difference methods for wave propagation in
  discontinuous media}}.
\bjtitle{Journal of Computational Physics}
\bvolume{220}(\bissue{1}),
\bfpage{249}--\blpage{269}
(\byear{2006}).
\doiurl{10.1016/j.jcp.2006.05.007}
\end{barticle}
\endbibitem

\end{thebibliography}

\section*{Statements and Declarations}
\noindent\textbf{Funding } The author did not receive support from any organization for the submitted work.
\\
\noindent\textbf{Conflict of interest } The author has no conflicts of interest to declare that are relevant to the content of this article.

\noindent\textbf{Data availability } Data sharing not applicable to this article as no datasets were generated or analyzed during the current study.

\end{document}